\documentclass[12pt]{article}

\usepackage{amssymb}
\usepackage{amsfonts}
\usepackage{amsmath,amsthm}

\date{\today}

\newtheorem{theorem}{Theorem}
\newtheorem{proposition}{Proposition}
\newtheorem{definition}{Definition}
\newtheorem{example}{Example}
\newtheorem*{Gordan}{Gordan's Theorem of the Alternative}{\bf}{\it}
\newtheorem*{Mozkin}{Motzkin's Theorem of the Alternative}{\bf}{\it}
\newtheorem*{th2.4}{Theorem 2.4 in Ref.\;\cite{osu99}}{\bf}{\it}
\newtheorem*{th3.6}{Theorem 3.6 in Ref.\;\cite{osu99}}{\bf}{\it}
\newtheorem*{KT}{Kuhn-Tucker Theorem}{\bf}{\it}
\newcommand{\R}{\mathbb R}

\title{Multiobjective Programming and Weighting Scalar Problem\thanks{This research is partially supported by the TU Varna Grant No 18/2012.}}

\author{\footnotesize VSEVOLOD I. IVANOV\thanks{Department of Mathematics, Technical University of Varna, Bulgaria. Email: vsevolodivanov@yahoo.com}}

\begin{document}

\maketitle
\begin{abstract}
In this work is  discussed the paper by Osuna-Gomez, Beato-Moreno and Rufian-Lizana, Generalized convexity in multiobjective programming, {\it J. Math. Anal. Appl.} 233 (1999) 205--220. We point out an error in the proofs of main Theorems 2.4 and 3.6 in this reference. Then we derive correct proofs, which  are based on a different approach, because the scheme of the proofs of the authors is not applicable. As a result of the discussion two additional theorems are obtained such that the scheme of the proofs of Theorems 2.4 and 3.6 in the mentioned reference works for them.


\end{abstract}

\section{Introduction}
\label{s1}
Generalized convexity assumptions play very important role in scalar characterizations of optimality in multiobjective programming. Invex functions are among the most general classes of functions. Several classes of generalized convex ones satisfy the property that every stationary point is a global minimizer, but invex functions are the largest class with this property. Invex scalar functions were introduced by Hanson \cite{han81}. 
After Hanson's initial work appeared the paper of  Martin \cite{mar85}, where KT-invex scalar problems with inequality constraint   were introduced. They are the largest class of problems such that every point, which satisfies Kuhn-Tucker optimality conditions, is a global minimizer. Martin's optimality conditions were generalized by   Osuna-Gomez, Beato-Moreno and Rufian-Lizana   \cite{osu99} to vector functions and vector problems.

In the paper \cite{osu99} has been studied the weighting problem. In the present work, we discuss Theorems 2.4 and 3.6 from the Ref. \cite{osu99}. We point out an error in the proofs of these theorems. A simple example shows that their proofs are not correct. After that we derive correct proofs of the mentioned  results, which are based on a different approach, because the scheme of the proofs of the authors is not applicable. Following the scheme of the proofs of the authors of Theorems 2.4 and 3.6 in \cite{osu99}, we obtain two additional theorems, which concern the connection between the vector problem and its scalarization.

\section{Preliminaries}

In the present paper, we consider the following problem:
\medskip

minimize$\quad f(x)=(f_1(x),f_2(x),\dots,f_n(x))$

subject to$\quad x\in S\subseteq\R^s$. \hfill  (P)
\medskip

First, we suppose that $S$ is an open set in the space $\R^s$. Next, we suppose that $S$ is functionally specified by inequality constraints.
All results given here are obtained for Fr\'echet differentiable  functions. 

We begin with some preliminary definitions and notations. We denote the scalar product of the vectors $a$ and $b$ by $a\, b$ and the Jacobean matrix of the vector function $f$ at the point $x$ by $J f(x)$. We use the following standard notation comparing vectors $a=(a_1,a_2,\dots,a_n)\in\R^n$ and $b=(b_1,b_2,\dots,b_n)\in\R^n$:

\bigskip \noindent
\begin{tabular}{@{}lllll}
$a<b$ & iff & $a_i<b_i$ & for all & $i=1,2,\dots,n$;\\
$a\leqq b$ & iff & $a_i\le b_i$ & for all & $i=1,2,\dots,n$; \\
$a\le b$ & iff & $a_i\le b_i$ & for all & $i=1,2,\dots,n$, but $a\ne b$.
\end{tabular}
\bigskip

The following theorems of the alternative are well-known:
\begin{Gordan}{\bf(\cite{gor1873,man69})}
For each given matrix $A$, either the system
\[
Ax<0
\]
has a solution $x$, or the system
\[
A^Ty=0,\; y\ge 0,\; y\ne 0
\]
has a solution $y$, but never both.
\end{Gordan}

\begin{Mozkin}{\bf(\cite{moz1936,man69})}
Let $A$ and $B$ be given matrices, with $A$ being non vacuous. Then either the system
\[
Ax<0,\quad Bx\leqq 0
\]
has a solution $x$, or the system
\[
A^Ty+B^Tz=0,\; y\ge 0,\; z\geqq 0
\]
has a solution $(y,z)$, but never both.
\end{Mozkin}

\begin{definition}
A point $\bar x\in S$ is called a weakly efficient solution of {\rm (P)} iff there is no $x\in S$ such that $f(x)<f(\bar x)$.
\end{definition}
                                                                                                               
\begin{definition}[\cite{osu99}]
A point $(\bar x,\bar\lambda)\in S\times\R^n$ is called critical for the vector function $f$ (or vector critical point) iff $\bar\lambda\geqq 0$, $\bar\lambda\ne 0$, and $\bar\lambda\, J f(\bar x)=0$. 
\end{definition}

A fundamental method for solving the multiobjective problem (P) is the linear scalarization. For every vector $\lambda$ from the set 
\[
\Lambda:=\{\lambda=(\lambda_1,\lambda_2,\ldots,\lambda_n)\in\R^n\mid\sum_{i=1}^n\lambda_i=1,\;\lambda_i\geqq 0,\; i=1,2,\ldots, n\}
\]
we consider the following weighting scalar problem:

\bigskip
minimize$\quad \lambda f(x)\quad$subject to$\quad x\in S$.\hfill ${\rm (P_\lambda)}$
\bigskip

The following well-known results shows why linear scalarization is very important:
\begin{proposition}
If $\bar x$ is a weakly efficient solution of the problem {\rm (P)}, then there exists $\bar\lambda\in\Lambda$ such that $(\bar x,\bar\lambda)$ is a vector critical point for $f$.
\end{proposition}

\begin{proposition}[\cite{geo68}]
Let $\bar x\in S$ be a global solution of ${\rm (P_\lambda)}$ for some $\lambda\in\Lambda$. Then $\bar x$ is a weakly efficient solution of ${\rm (P)}$.
\end{proposition}

Another result concerning weighting problem of (P) and invex functions is the following one:

\begin{proposition}[\cite{osu99}, Theorem 2.3]\label{pr1}
Let $f$ be an invex vector function, defined on some open set $S$. If $x,$ is a vector critical point and $\lambda$ is the respective multiplier such that $\lambda\ge 0$ and $\lambda\, J f(x)=0$, then $x$ solves the weighting problem {\rm (P$_{\lambda}$)}.
\end{proposition}

\section{Unconstrained multiobjective programming problems}
In this section, we consider the problem (P) with $S$ being an open set.

\begin{definition}
A  Fr\'echet differentiable vector function $f$, defined on some open set $S\subseteq\R^s$ is called invex  iff for all $x\in S$ and $y\in S$ there exists a vector $\eta(x,y)\in\R^s$ which depend on $x$ and $y$ such that
\[
f(y)-f(x)\geqq J f(x)\,\eta(x,y)
\]
\end{definition}


\begin{th2.4}\label{th1}
Each vector critical point is a weakly efficient solution and solves a weighting scalar problem if and only if the objective function is invex.
\end{th2.4}

It is not clear from this statement whether the multiplier $\lambda$, which appears in the definition of a vector critical point, is the same as the multiplier, which appears in the weighting problem. One can see from the proof in Ref. \cite{osu99} that the authors suppose that the multiplier is the same. In the proof of the theorem, the authors conclude from the assumption that the set of vector critical points, the set of weakly efficient global solutions and the set of global solutions of the weighting problem coincide, that for all $x\in S$, $\bar x\in S$ the system
\begin{equation}\label{2}
\left[
\begin{array}{l|l}
0 & 1 \\
J f(\bar x) & f(x)-f(\bar x)
\end{array}\right]\left[
\begin{array}{l}
\eta \\
\xi
\end{array}\right]<0
\end{equation}
has a solution $(\eta,\xi)$, $\eta\in\R^s$, $\xi\in\R$, applying Gordan's theorem of the alternative. The following example disprove this fault conclusion. Therefore, the proof of the theorem is not correct.

\begin{example}\label{ex1}
Consider the vector function $f:\R\to\R^2$, $f=(f_1,f_2)$ defined as follows:
\[
f_1(x)=\left\{
\begin{array}{ll}
(x+1)^2, & \text{if}\quad x<-1 \\
0, & \text{if}\quad -1\le x\le 1 \\
(x-1)^2, & \text{if}\quad x>1
\end{array}\right.
\quad\text{and}\quad
f_2(x)=\left\{
\begin{array}{ll}
(x+1)^4, & \text{if}\quad x<-1 \\
0, & \text{if}\quad -1\le x\le 1 \\
(x-1)^4, & \text{if}\quad x>1.
\end{array}\right.
\]
Both functions $f_1$ and $f_2$ are convex and everywhere differentiable. The set of weakly efficient points and the set of vector critical points coincide with the closed interval $[-1,1]$. For every $\lambda\ge 0$ the set of global solutions of the weighting problem also coincide with the interval $[-1,1]$. Let us take $\bar x=0$ and $x$ is an arbitrary point from the interval $[-1,1]$ such that $x\ne\bar x$. Then $J f(\bar x)=0$, $f(x)-f(\bar x)=0$ and the system (\ref{2})
has no solutions. 
\end{example}

Let us consider the proof of Theorem 2.4 in \cite{osu99}.
It follows from Gordan's Theorem of the alternative that the system (\ref{2}) has a solution $(\eta,\xi)\in\R^{s+1}$ for all $x\in S$, $\bar x\in S$ if and only if the system

\begin{equation}\label{3}
[\nu\mid\lambda]\left[
\begin{array}{l|l}
0 & 1 \\
J f(\bar x) & f(x)-f(\bar x)
\end{array}\right]=0,\; (\nu,\lambda)\geqq 0,\; (\nu,\lambda)\ne 0
\end{equation}
is not solvable for all $x\in S$, $\bar x\in S$. The last statement is quite different from
\[
[\nu\mid\lambda]\left[
\begin{array}{l|l}
0 & 1 \\
J f(\bar x) & f(x)-f(\bar x)
\end{array}\right]=0,\\
\nu>0,\;\lambda\geqq 0,\; \lambda\ne 0.
\]
The error in the proof of Theorem 2.4 in Ref. \cite{osu99} passes through the false statement, which does not appear in the proof, that for arbitrary $x\in S$ and $\bar x\in S$ the system
\[
\lambda\, J f(\bar x)=0,\;\lambda[f(x)-f(\bar x)]=0,\; \lambda\geqq 0,\;\lambda\ne 0
\]
has no solutions. This statement is a consequence of the assumption that every vector critical point with a respective multiplier $\lambda$ is the unique solution of the weighting problem. If we make such assumption, then the proof of Theorem 2.4 in Ref. \cite{osu99} become correct (see Theorem \ref{th2} below).

The next definition is a slight generalization of the notion of strictly invex scalar function.

\begin{definition}
We call a Fr\'echet differentiable vector function $f$, defined on some open set $S\subseteq\R^s$, strictly invex  iff for all $x\in S$ and $y\in S$, $y\ne x$ there exists a vector $\eta(x,y)\in\R^s$ which depend on $x$ and $y$ such that
\[
f(y)-f(x)>J f(x)\,\eta(x,y).
\]
\end{definition}

\begin{theorem}\label{th2}
Let the vector function $f$  be Fr\'echet differentiable on some open set $S\subseteq\R^s$. Then for each vector critical point $x$ with a respective multiplier $\lambda$ the point $x$ is the unique global solution of the weighting problem ${\rm (P_\lambda)}$ with the same vector $\lambda$ if and only if ${\rm (P)}$ is strictly invex.
\end{theorem}
\begin{proof}
Suppose that every vector critical point with a multiplier $\lambda$  is the unique global solution of the weighting problem ${\rm (P_\lambda)}$.
Take arbitrary points $x\in S$ and $\bar x\in S$, $x\ne \bar x$. Then the system
\[
\lambda\, J f(\bar x)=0,\;\nu+\lambda\,[f(x)-f(\bar x)]=0,\; \nu>0,\;\lambda\geqq 0
\]
for $\nu$ and $\lambda$ has no solutions, because every vector critical point is a global solution of the weighting problem. The system 
\[
\lambda\, J f(\bar x)=0,\;\lambda\,[f(x)-f(\bar x)]=0,\; \lambda\geqq 0,\;\lambda\ne 0
\]
for $\lambda$ also has no solutions, because every vector critical point is the unique global solution of the weighting problem. Therefore, the system (\ref{3})
has no solutions. Then it follows from Gordan's alternative theorem that the system
\[
\left[
\begin{array}{l|l}
0 & 1 \\
J f(\bar x) & f(x)-f(\bar x)
\end{array}\right]\left[
\begin{array}{l}
\zeta \\
\xi
\end{array}\right]<0
\]
has a solution $(\zeta,\xi)$, $\zeta\in\R^s$, $\xi\in\R$ for all $x\in S$, $\bar x\in S$. It follows from here that there exist $\xi<0$ and $\zeta\in\R^s$ such that
\[
\xi[f(x)-f(\bar x)]+J f(\bar x)\,\zeta<0.
\]
Let $\eta=(-\zeta)/\xi$. Then we have
\[
f(x)-f(\bar x)>J f(\bar x)\,\eta
\]
which implies that $f$ is strictly invex, because $x$ and $\bar x$ are arbitrary points from $S$.

Suppose that $f$ is strictly invex. We prove that for every vector critical point $\bar x$ with a respective multiplier $\bar\lambda$ the point $\bar x$ is the only solution of $(P_{\bar\lambda})$. Assume the contrary that there exist $\bar x\in S$, $x\in S$, $x\ne\bar x$ and $\bar\lambda\in\Lambda$ such that
\begin{equation}\label{1}
\bar\lambda\, J f(\bar x)=0,\quad 
\bar\lambda\,[f(x)-f(\bar x)]\le 0.
\end{equation}
It follows from strict invexity that there exists $\eta\in\R^s$ with
\[
f_i(x)-f_i(\bar x)>\nabla f_i(\bar x)\,\eta\quad\textrm{for all}\quad i=1,2,\dots,n.
\] 
Taking into account that at least one $\bar\lambda_i$ is strictly positive, we obtain that
\[
\bar\lambda f(x)-\bar\lambda f(\bar x)>\bar\lambda\, J f(\bar x)\,\eta,
\]
which is contrary to Equations (\ref{1}).
\end{proof}

Example \ref{ex1} does not prove that Theorem 2.4 in \cite{osu99} is not satisfied, because the functions $f_1$ and $f_2$ in this example are convex; hence invex with respect to the kernel $\eta(x,y)=y-x$. It only shows that the proof, which is given there is not correct. Really, the following theorem holds. It is equivalent to Theorem 2.4 in \cite{osu99}:

\begin{theorem}\label{th3}
Let ${\rm (P)}$  be Fr\'echet differentiable. Then each vector critical point $x$ with a respective multiplier $\lambda$ is a global solution of the weighting problem ${\rm (P_\lambda)}$ with the same vector $\lambda$ if and only if ${\rm (P)}$ is  invex.
\end{theorem}
\begin{proof}
Let each vector critical point $x$ with a respective multiplier $\lambda$ be a global solution of the weighting problem (P$_\lambda$). We prove that (P) is  invex. Suppose the contrary that there exist $x\in S$ and $\bar x\in S$ such that the system 
\[
\nabla f_i(\bar x)\,\eta\leqq f_i(x)-f_i(\bar x),\; i=1,2,\dots,n
\]
has no solution for $\eta\in\R^s$. Therefore, the linear programming problem 

\medskip
\begin{tabular}{ll}
maximize & 0         \\
subject to &$J f(\bar x)\,\eta \leqq f(x)-f(\bar x)$
\end{tabular}
\medskip

\noindent
is infeasible. Its dual linear programming problem is the following one:

\medskip
\begin{tabular}{ll}
minimize & $\lambda\,[f(x)-f(\bar x)]$ \\
subject to &
$\lambda\, J f(\bar x)=0$, \\
& $\lambda\geqq 0$.
\end{tabular}
\medskip

It follows from Duality theorem in linear programming that both problems are simultaneously solvable or unsolvable. The dual problem is feasible, because $\lambda=0$ is a feasible point.  Therefore, the dual problem is not solvable, because it is unbounded. Hence, there exists a feasible point $\bar\lambda$ for the dual problem such that $\bar\lambda\ne 0$ and  $\bar\lambda\,[f(x)-f(\bar x)]<0$. On the other hand, since $\bar\lambda$ is a feasible point, then $\bar x$  is a vector critical point for $f$ with a respective multiplier $\bar\lambda$. According to the assumption $\bar x$ is a solution of the weighting problem (P$_{\bar\lambda}$) we have $\bar\lambda\,[f(x)-f(\bar x)]\geqq 0$, which is a contradiction.

The converse claim is Proposition \ref{pr1} (Theorem 2.3 in \cite{osu99}).
\end{proof}

\section{Weighting scalar problem and  Kuhn-Tucker stationary points}

In this section, we consider the case when (P) is an inequality constrained problem, that is 
\[
S:=\{x\in X\mid g(x)\leqq 0\},
\]
where $X\subseteq\R^s$ is an open set and $f:X\to\R^n$, $g:X\to\R^m$.
For every feasible point $x\in S$,  let $I(x)$ be the set of active constraints
\[
I(x):=\{i\in\{1,2,\ldots,m\}\mid g_i(x)=0\}.
\]

The following theorem is well-known (see, for example, \cite[Theorem 6.6.10]{gio04}):

\begin{KT}
Let the problem {\rm (P)} with inequality constraints be  Fr\'echet differentiable and $\bar x\in S$ be a weak local
minimizer. Suppose that the Kuhn-Tucker constraint qualification holds. Then there exist Lagrange
multipliers $\bar\lambda=(\bar\lambda_1,\bar\lambda_2,\dots,\bar\lambda_n)\ge 0$, $\bar\mu=(\bar\mu_1,\bar\mu_2,\ldots,\bar\mu_m)\geqq 0$, with
\begin{equation}\label{KT-1}
\mu_i g_i(\bar x)=0,\;i=1,2,\ldots,m,
\end{equation}
\begin{equation}\label{KT-2}
\sum_{i=1}^n\bar\lambda_i\nabla f_i(\bar x)+\sum_{i\in I(\bar x)}\bar\mu_i\nabla g_i(\bar x)=0.
\end{equation}
\end{KT}

\begin{definition}
Each point $x\in S$ which satisfies these necessary optimality conditions with
$\lambda\ge 0$ is called a Kuhn-Tucker stationary point (for short, KT point).
\end{definition}

The connection between Kuhn-Tucker stationary points and weighting problem has been studied in \cite[Theorem 3.6]{osu99} and its proof is again not correct.
For every vector $\lambda\in\Lambda$ consider the weighting scalar problem (P$_\lambda$). The following notion was introduced by Osuna-Gomez, Beato-Moreno, Rufian-Lizana \cite{osu99}. 

\begin{definition}\label{df7}
Let the  problem {\rm (P)} be Fr\'echet differentiable.
Then it is called KT-invex iff  there exists a map $\eta: X\times X\to\R^s$ 
such that the following implication holds:
\[
\left.
\begin{array}{l}
x\in X,\; y\in X,\;   \\
g(x)\leqq 0,\; g(y)\leqq 0 \\
\end{array}\right]
\quad\Rightarrow \quad
\left[
\begin{array}{l}
 f(y)-f(x)\geqq J f(x)\,\eta(x,y) \\
0\geqq\nabla g_{j}(x)\eta(x,y),\; j\in I(x)
\end{array}
\right.
\]
\end{definition}

\begin{th3.6}
Every vector Kuhn-Tucker point is a weakly efficient point and solves a weighting scalar problem if and only if the problem {\rm (P)} is KT-invex.
\end{th3.6}
The error in the proof of this theorem in \cite{osu99} is similar to the error in Theorem 2.4. It is an incorrect application of the Motzkin's theorem of the alternative.

We introduce the following definition:

\begin{definition}
Let the  problem {\rm (P)} be Fr\'echet differentiable.
Then we call {\rm (P)} strictly KT-invex iff  there exists a map $\eta: X\times X\to\R^s$ 
such that the following implication holds:
\[
\left.
\begin{array}{l}
x\in X,\; y\in X,\;y\ne x   \\
g(x)\leqq 0,\; g(y)\leqq 0 \\
\end{array}\right]
\quad\Rightarrow \quad
\left[
\begin{array}{l}
 f(y)-f(x)> J f(x)\,\eta(x,y) \\
0\geqq\nabla g_{j}(x)\eta(x,y),\; j\in I(x)
\end{array}
\right.
\]
\end{definition}

\begin{theorem}\label{th4}
Let ${\rm (P)}$  be Fr\'echet differentiable. Then a  KT stationary point $\bar x$ such that for every multiplier $(\bar\lambda,\bar\mu)$, which satisfies the Kuhn-Tucker necessary optimality conditions (\ref{KT-1}) and (\ref{KT-2}) together with $\bar x$
the point $\bar x$ is the unique global solution of the weighting problem ${\rm (P_{\bar\lambda})}$ with the same vector $\bar\lambda$
if and only if ${\rm (P)}$ is strictly KT-invex.
\end{theorem}
\begin{proof}
Let (P) be strictly KT-invex and $\bar x$ be a  KT stationary point. Then there exist vectors $\bar\lambda$ and $\bar\mu$ which satisfy together with $\bar x$ Equations (\ref{KT-1}), (\ref{KT-2}). 
We prove that  $\bar x$ is the unique global solution of the respective weighting problem (P$_{\bar\lambda}$). Choose arbitrary feasible point $x\in S$, 
$x\ne\bar x$.  It follows from strict KT-invexity that there exists  a vector $\eta\in\R^s$ such that 
\begin{equation}\label{5}
f_i(x)-f_i(\bar x)>\nabla f_i(\bar x)\,\eta(x,\bar x),\quad i=1,2,\dots,n
\end{equation}
\begin{equation}\label{6}
0\geqq\nabla g_{j}(\bar x)\,\eta(x,\bar x),\; j\in I(\bar x).
\end{equation}
Let us multiply (\ref{5}) by $\bar\lambda_i$, (\ref{6}) by $\bar\mu_j$ and add the obtained inequalities. Then it follows from (\ref{KT-1}), (\ref{KT-2}) 
and $\bar\lambda\ne 0$ that
\[
\bar\lambda f(x)-\bar\lambda f(\bar x)>0.
\]
Since $x$ is an arbitrary feasible point, then $\bar x$ is the only global solution of (P$_{\bar\lambda}$).

Let every Kuhn-Tucker stationary point $\bar x$, which satisfy together with some vectors $\bar\lambda$ and $\bar\mu$ Equations  (\ref{KT-1}), (\ref{KT-2}) 
is the only global solution of (P$_{\bar\lambda}$). Denote by $J g_I(\bar x)$ the matrix, whose columns are the vectors $\nabla g_j(\bar x)$, $j\in I(\bar x)$. Then for arbitrary feasible point $x$, $x\ne \bar x$ the system
\[
\lambda\, J f(\bar x)+\mu\, J g_I(\bar x)=0,\;\nu+\lambda\,[f(x)-f(\bar x)]=0,\; \nu>0,\;\lambda\geqq 0,\;\mu\geqq 0
\]
for $\nu$, $\lambda$ and $\mu$ has no solutions, because every Kuhn-Tucker stationary point is a global solution of the weighting problem. The system 
\[
\lambda\, J f(\bar x)+\mu\, J g_I(\bar x)=0,\;\lambda\,[f(x)-f(\bar x)]=0,\; \lambda\geqq 0,\;\lambda\ne 0,\;\mu\geqq 0
\]
for $\lambda$ and $\mu$ also has no solutions, because every  Kuhn-Tucker stationary point is the only global solution of the weighting problem. Therefore, the system
\begin{equation}\label{7}
[\nu\mid\lambda]\left[
\begin{array}{l|l}
0 & 1 \\
J f(\bar x) & f(x)-f(\bar x)
\end{array}\right]+\mu\left[J g_I(\bar x)\mid 0\right]
=0,\; (\nu,\lambda,\mu)\geqq 0,\; (\nu,\lambda)\ne 0
\end{equation}
has no solutions. Then it follows from Motzkin's theorem of the alternative that for all $x\in S$, $\bar x\in S$, $x\ne\bar x$ the system
\[\left\{
\begin{array}{r}
\left[
\begin{array}{l|l}
0 & 1 \\
J f(\bar x) & f(x)-f(\bar x)
\end{array}\right]\left[
\begin{array}{l}
\zeta \\
\xi
\end{array}\right]<0 \\
\left[J g_I(\bar x)\mid 0\right]\left[
\begin{array}{l}
\zeta \\
\xi
\end{array}\right]\leqq 0
\end{array}\right.
\]
has a solution $(\zeta,\xi)$, $\zeta\in\R^s$, $\xi\in\R$. It follows from here that there exist $\xi<0$ and $\zeta\in\R^s$ such that
\[
\xi\,[f(x)-f(\bar x)]+ J f(\bar x)\,\zeta<0,\quad J g_I(\bar x)\,\zeta\leqq 0.
\]
Let $\eta=(-\zeta)/\xi$. Then we have
\[
f(x)-f(\bar x)> J f(\bar x)\,\eta,\quad 0\geqq J g_I(\bar x)\,\eta,
\]
which implies that $f$ is strictly invex, because $x$ and $\bar x$ are arbitrary points from $S$ such that $x\ne\bar x$.
\end{proof}

\begin{proposition}[\cite{osu99}]\label{th3.5}
Let the problem {\rm (P)} be KT-invex.  Then for every Kuhn-Tucker stationary triple $(x,\lambda,\mu)$ the point $x$ is a global solution of the respective weighting scalar problem {\rm (P$_\lambda$)}.
\end{proposition}

The next result is equivalent to \cite[Theorem 3.6]{osu99} with a corrected proof:

\begin{theorem}\label{th5}
Let ${\rm (P)}$  be Fr\'echet differentiable. Then each  KT stationary point $x$ such that for every multiplier $(\lambda,\mu)$, which satisfies the Kuhn-Tucker necessary optimality conditions together with $x$, the point $x$ is a global solution of the weighting problem ${\rm (P_\lambda)}$ with the same vector $\lambda$
if and only if ${\rm (P)}$ is  KT-invex.
\end{theorem}
\begin{proof} 
Let each KT stationary point be a solution of the weighting problem. We prove that (P) is KT-invex. Suppose the contrary that there exist $x\in S$, $\bar x\in S$ such that the system
\[
\left\{
\begin{array}{l}
J f(\bar x)\,\eta\leqq f(x)-f(\bar x) \\
\nabla g_{j}(x)\,\eta\leqq 0,\; j\in I(x)
\end{array}\right.
\]
has no solutions for $\eta\in\R^s$. Therefore, the linear programming problem 

\medskip
\begin{tabular}{ll}
maximize & 0         \\
subject to &$\nabla f_i(\bar x)\,\eta\leqq f_i(x)-f_i(\bar x)$, $i=1,2,\ldots,n$,  \\
& $\nabla g_j(\bar x)\,\eta \leqq 0,\; j\in I(\bar x)$
\end{tabular}
\medskip

\noindent
is infeasible. Its linear programming dual problem is the following one:

\medskip
\begin{tabular}{ll}
minimize & $\sum_{i=1}^n\lambda_i\, [f_i(x)-f_i(\bar x)]$ \\
subject to &
$\sum_{i=1}^n\lambda_i\nabla f_i(\bar x)+\sum_{j\in I(\bar x)}\mu_j\nabla g_j(\bar x)=0$, \\
& $\lambda_i\geqq 0,\; i=1,2,\dots, n,\quad  \mu_j\geqq 0,\; j\in I(\bar x)$.
\end{tabular}
\medskip

It follows from Duality theorem that both problems are simultaneously solvable or unsolvable. The dual problem is feasible, because $\lambda=0$, $\mu=0$ is a feasible point.  Therefore, the dual problem is not solvable and it is unbounded. Hence, there exists a feasible point $(\lambda,\mu)$ for the dual problem such that $\lambda\ne 0$. Therefore, $\bar x$ is a Kuhn-Tucker stationary point. According to the hypothesis $\bar x$ is a solution of the weighting problem, that is $\lambda\,[f(x)-f(\bar x)]\geqq 0$. This conclusion is satisfied for every $(\lambda,\mu)$ with $\lambda\ne 0$. Hence the optimal value of the dual problem is bounded from below by 0, which contradicts our indirect conclusion that it is unbounded.

The inverse claim is Proposition \ref{th3.5} (Theorem 3.5 in Ref. \cite{osu99}).
\end{proof}

\end{document}